\documentclass[12pt]{amsart}
\usepackage{fullpage}

\usepackage{amsmath}
\usepackage{amssymb}
\usepackage{amsfonts}
\usepackage{graphicx}
\usepackage{amsthm}
\usepackage{enumerate}
\usepackage{lscape}
\usepackage{dsfont}
\usepackage{color}
\usepackage{mathtools}
\usepackage[dvipsnames]{xcolor}
\usepackage{hyperref}

\usepackage{tikz-cd}

\newcounter{example}
\newenvironment{example}[1][]{\refstepcounter{example}\par\medskip
\noindent \textbf{Example~\theexample. (#1)} \rmfamily}{\medskip}    

\usepackage[utf8]{inputenc}

\usepackage{apptools}
\newcommand{\C}{{\mathbb C}}

\newcommand{\R}{{\mathbb R}}

\newtheorem{theorem}{Theorem}

\newtheorem{defin}{Definition}

\newtheorem{cor}{Corollary}

\newtheorem{remark}{Remark}

\makeatletter
\@namedef{subjclassname@2020}{%
  \textup{2020} Mathematics Subject Classification}
\makeatother

\theoremstyle{definition}

\makeatletter
\def\@tocline#1#2#3#4#5#6#7{\relax
  \ifnum #1>\c@tocdepth 
  \else
    \par \addpenalty\@secpenalty\addvspace{#2}%
    \begingroup \hyphenpenalty\@M
    \@ifempty{#4}{%
      \@tempdima\csname r@tocindent\number#1\endcsname\relax
    }{%
      \@tempdima#4\relax
    }%
    \parindent\z@ \leftskip#3\relax \advance\leftskip\@tempdima\relax
    \rightskip\@pnumwidth plus4em \parfillskip-\@pnumwidth
    #5\leavevmode\hskip-\@tempdima
      \ifcase #1
       \or\or \hskip 1em \or \hskip 2em \else \hskip 3em \fi%
      #6\nobreak\relax
    \hfill\hbox to\@pnumwidth{\@tocpagenum{#7}}\par
    \nobreak
    \endgroup
  \fi}
\makeatother
\makeatletter

\numberwithin{equation}{section}

\def\subsubsection{\@startsection{subsubsection}{3}%
\z@{.5\linespacing\@plus.7\linespacing}{-.5em}%
{\normalfont\bfseries}}
\makeatother

\title{CR-relatives K\"ahler manifolds}

\author{Stefano Marini}
\address{(Stefano Marini) 
Universit\`a di Parma (Italy)\\  Dipartimento di Scienze Matematiche, Fisiche e Informatiche}
        \email{stefano.marini@unipr.it}

\author{Michela Zedda}
\address{(Michela Zedda) Universit\`a di Parma (Italy)\\  Dipartimento di Scienze Matematiche, Fisiche e Informatiche}
\email{michela.zedda@unipr.it}
\date{\today}

\thanks{
This work has been financially supported by the Programme “FIL-Quota Incentivante” of University of Parma, co-sponsored by Fondazione Cariparma, by the PRIN project ``Real and Complex Manifolds: Topology, Geometry and holomorphic dynamics'', and by the group G.N.S.A.G.A. of I.N.d.A.M.} 

\begin{document}

\maketitle

\begin{abstract}
In this paper we show that two K\"ahler manifolds which do not share a K\"ahler submanifold, do not share either a Levi degenerate CR--submanifold with constant dimension Levi kernel. In particular, they do not share a CR--product.  Further, we obtain that a Levi degenerate CR-submanifold of $\mathds C^n$ cannot be isometrically immersed into a flag manifold.

\end{abstract}

\section{Introduction and statement of the main result}

Two K\"ahler manifolds $({\rm\tilde M_1},\omega_1)$, $({ \rm\tilde M_2}, \omega_2)$, are called {\em relatives} when they share a common K\"ahler submanifold, that is when there exists a K\"ahler manifold $(\rm M,\omega)$ and two holomorphic maps $f_1\!:\rm M\rightarrow \tilde M_1$, $f_2\!:\rm M,\rightarrow  {\rm\tilde M_2}$, such that $f_1^*\omega_1=f_2^*\omega_2$.
The problem of understanding when two K\"ahler manifolds are relatives appears in literature in Umehara's work \cite{umehara}, where it is proven that two complex space form with holomorphic sectional curvature of different signs can not be relative. Umehara's proof relies on Calabi's criterion for the existence of a holomorphic and isometric immersion of a K\"ahler manifold into complex space forms \cite{calabi} (see also \cite{libro}). In \cite{loidiscala} it is shown that Hermitian symmetric spaces with different compact types are not relatives. Further, in \cite{huangyuan} it has been proven that Euclidean spaces and Hermitian symmetric spaces of noncompact type are not relatives, and the same result for Hermitian symmetric spaces of compact type follows by Umehara's result and Nakagawa-Takagi embedding of Hermitian symmetric spaces of compact type into the complex projective space. The relativity problem has been investigated also in \cite{chengdiscalayuan}, where it has been reformulated in terms of the Umehara's algebra, and in \cite{mossa, sutangtu, relatives}.

In this paper we are interested in studying when two K\"ahler manifolds are {\em CR--relatives}, i.e. when they share a CR--submanifold.  The notion of CR--submanifold of a K\"ahler manifold was introduced and firstly studied by A. Bejancu in \cite{bejancuI,bejancuII}.  In \cite{chenI} B. Y. Chen studied a special family of CR--submanifolds into complex space forms called CR--products. In particular, he proved that a CR--product isometrically immersed in $\mathds C^m$ is globally a Riemannian product of a holomorphic submanifold and a totally real one, and he gave a characterization of CR--products of maximal dimension isometrically immersed in a complex projective space. In \cite{chenII} the existence of a foliation with totally real leaves is investigated, and further results on CR--submanifolds of complex space forms can be found in \cite{bejancuIII}.  For a more comprehensive exposition on this topic, we refer the reader to \cite{yanokon1982}.  Our main result is the following:

\begin{theorem}\label{main}
Two K\"ahler manifolds which are not relatives do not share a Levi  degenerate CR--submanifold with constant dimension Levi kernel.
\end{theorem}

Notice that the Levi degenerate condition cannot be dropped, as shown for example by the canonical CR-structure on the $3$-sphere, which is induced from both the flat K\"ahler structure on $\mathds C^2$ and the Fubini--Study one on $\mathds C{\rm P}^2$, proving that $\mathds C^2$ and $\mathds C{\rm P}^2$ are an example of CR--relatives K\"ahler manifolds which, by Umehara's result Th. \ref{umehara} below, are not relatives.  It is worth pointing out that if two K\"ahler manifolds are relative then they are also CR--relatives in a trivial way, since a K\"ahler submanifold can be viewed as a (holomorphic) CR--submanifold. Finally, the condition on the Levi kernel is a natural assumption which is fundamental to obtain the existence  of a foliation with complex leaves.

As direct consequence of Theorem \ref{main} we get the following:
\begin{cor}\label{cor2}
K\"ahler manifolds which are not relatives do not share a CR--product.
\end{cor}

Further, combining Theorem \ref{main} with a result of A. Loi and R. Mossa (see \cite{loimossa} or Th. \ref{loimossa} below) we get the following:
\begin{cor}\label{cor}
A Levi degenerate with constant dimension Levi kernel CR-submanifold of $\mathds C^m$ cannot be isometrically immersed into a flag manifold.
\end{cor}
In particular the CR--submanifolds of hypersurface type of $\mathds C^m$ given in Example \ref{cartan} cannot be isometrically immersed as CR--submanifolds into a flag manifolds. \\

The paper is organized as follows. In Section  \ref{relsec} we collects some results about K\"ahler relatives and K\"ahler immersions into the complex flat and complex projective spaces,  while in Section \ref{crsec} we summarize what we need on CR--manifolds and CR--submanifolds of a K\"ahler manifold.  Finally in Section \ref{proofmain} we prove Theorem \ref{main} and its corollaries.
\section{Preliminaries on K\"ahler manifolds}\label{relsec}
Throughout this section, let ${\rm M}$ be a K\"ahler manifold endowed with a K\"ahler metric $g$, and let us denote by $\omega$ the closed $(1,1)$-form associated to $g$, i.e. $\omega(,\cdot,\cdot)=g(J\cdot,\cdot)$.  Recall that a K\"ahler manifold is characterized by the existence in a neighborhood $U$ of each point $x\in {\rm M}$ of a K\"ahler potential $\varphi\!:U\rightarrow \mathds R$ such that $\omega|_U:=\frac i2\partial\bar\partial \varphi$.

Denote by $g_0$ the flat metric on the complex euclidean space $\mathds C^m$, and by $\omega_0$ the associate K\"ahler form, i.e. $\omega_0:=\frac i2\partial\bar\partial ||z||^2$ (where $z=(z_1,\dots, z_m)$ are holomorphic coordinates on $\mathds C^m$). Further denote by $\mathds C{\rm P}^m_b=(\mathds C{\rm P}^m,g_b)$ the complex projective space endowed with the Fubini--Study metric $g_b$ of holomorphic sectional curvature $4b>0$. More precisely, given homogeneous coordinates $[Z_0,\dots, Z_{m}]$ on $\mathds C{\rm P}^m$ and setting affine coordinates $(z_1,\dots, z_m)$, with $z_j:=Z_j/Z_0$, $j=1,\dots, m$, on the open set $U_0=\{Z_0\neq 0\}$, a K\"ahler potential for the metric $g_b$ is given by:
\begin{equation}\label{phib}
\Phi_b(z):=\frac 1b\log\left(1+b\sum_{j=1}^m|z_j|^2\right).
\end{equation}
We denote by $\omega_b$ the K\"ahler form associated to $g_b$, i.e. $\omega_b|_{U_0}=\frac i2\partial \bar \partial \Phi_b$. When $b=1$ we usually use the notation $g_{FS}$ and $\omega_{FS}$ for $g_1$ and $\omega_1$ respectively.

We say that a K\"ahler immersion is {\em full} if the submanifold is not contained in any totally geodesic submanifold of the ambient space, that is in our context when the manifold ${\rm M}$ K\"ahler immersed in $\mathds C^m$ (or $\mathds C{\rm P}_b^m$) is not contained into $\mathds C^{m'}$ (or $\mathds C{\rm P}^{m'}_b$) for any $m'< m$. When the ambient space is a complex space form, E. Calabi  \cite{calabi} proved that a full K\"ahler immersion determines uniquely the dimension of the ambient space. In particular for $\mathds C^m$ and $\mathds C{\rm P}^m$ Calabi's result reads as follows.

\begin{theorem}[Calabi Rigidity Theorem \cite{calabi}]\label{calabirigidity}
If a neighborhood $V$ of a point $p$ admits a full Kähler immersion into $\mathds C^m$ or $\mathds C{\rm P}^m$, then $m$ is uniquely determined by the metric and the immersion is unique up to rigid motions of $\mathds C^m$ (respectively $\mathds C{\rm P}^m$).
\end{theorem}

Calabi Rigidity Theorem is a very useful tool to study relatives K\"ahler manifolds. The following result of M. Umehara \cite{umehara} together with Calabi's Rigidity proves that $\mathds C^m$ and $\mathds C{\rm P}^{m'}$ are not relatives for any $m$, $m'<\infty$.

\begin{theorem}[M. Umehara \cite{umehara}]\label{umehara}
Any K\"ahler submanifold of $\mathds C^m$, $m\leq \infty$, admits a full K\"ahler immersion into $\mathds C{\rm P}_b^\infty$, for any value of $b>0$.
\end{theorem}

Finally, we recall that up to rescale the metric by a positive factor, a homogeneous compact K\"ahler manifold admits a local  K\"ahler immersion into $\mathds C{\rm P}_b^m$, $m<\infty$.

\begin{theorem}[A. Loi, R. Mossa \cite{loimossa}]\label{loimossa}
Let $({\rm M},g)$ be a simply-connected homogeneous K\"ahler manifold with associated K\"ahler form $\omega$ integral. Then there exists a positive real number $\lambda$ such that $({\rm M},\lambda g)$ admits a K\"ahler immersion in $(\mathds C{\rm P}^m, g_{FS})$.
\end{theorem}
Observe that if we drop the hypothesis for ${\rm M}$ to be simply connected, Theorem \ref{loimossa} still applies locally. For details and further results about K\"ahler immersions of K\"ahler manifolds into complex space forms we refer the reader to \cite{libro} and references therein.

\section{Preliminaries on Cauchy--Riemann manifolds}\label{crsec}
For the arguments in this section we refer the reader to \cite{tomCR}. Let $\rm M$ be a smooth manifold of real dimension $2n+k$. A  CR--structure $(\rm M,  \mathcal{D}, \rm J)$ of type $(n, k)$ on $\rm M$ consists of the data of a smooth subbundle $\mathcal{D}\subseteq T\rm M$ of fiber dimension $2n$ and a fiber preserving smooth vector bundle isomorphism $\rm J : \mathcal{D}\rightarrow \mathcal{D},$ satisfying $\rm J^2 = -{\rm Id}$ and the integrability condition on the Nijenhuis tensor:
$$
\rm N(X,Y)=[X,Y]-[\rm J X,\rm J Y]-\rm J([\rm J X,Y]+[X,\rm J Y])=0, \quad \forall\  X,Y\in\Gamma(\rm M, \mathcal D).
$$
The numbers $n$ and $k$ are called respectively the CR--dimension and CR--codimension of $\rm M$. When $k=0$ $\rm M$ is {\em holomorphic}, while if $n=0$ it is called {\em totally real}.
 Prime examples of  CR--structures of CR--dimension $n$ are provided by the
real submanifolds $\rm M$ of a complex manifold $\tilde{\rm M}$ for which 
$
\dim_{\mathbb{C}} (T{\rm M}_x\otimes \mathbb{C}\cap T_x^{1,0}\tilde {\rm M})=n \,\,\,\forall x\in \rm M.
$
In this case $\mathcal{D}:= \mathcal{R}e(T{\rm M}\otimes\mathbb{C}\cap T^{0,1}\tilde{\rm M})={\rm J}T{\rm M}\cap T{\rm M}$
is a smooth real subbundle (of real  fiber dimension $2n$) of $T\rm M$.

The \emph {real  vector valued  Levi form} 
 is the real bilinear form: 
\begin{equation}\label{LR}
\mathcal{L}^{\R}:\mathcal{D}\times\mathcal{D}\rightarrow T{\rm M}\diagup \mathcal D, \quad \mathcal{L}^{\R} (X,Y):=\pi^\R([{\rm J}X,Y]-[X,{\rm J}Y])
\end{equation} 
where $\pi^\R$ is the canonical projection $T{\rm M}\rightarrow T{\rm M}\diagup\mathcal{D}.$
Observe that  
\eqref{LR} measures whether the subbundles $\mathcal{D}$ 
of $T\rm M$ is  integrable and
how the complex structure $\rm J$ interplays with the  integrability 
on the respective  fibers.
Define the \emph{Levi Kernel} at $x\in M $ as the null space of the Levi form:
\begin{equation}
\emph{Null}(\mathcal{L}^\R)=\{X\in\Gamma({\rm M},\mathcal{D})\,|\,\mathcal{L}^{\R}(X,Y)=0, \ \forall \, Y\in\Gamma({\rm M} ,\mathcal{D})  \}.
\end{equation}
We recall here the definition of Levi degenerate CR--manifold, recent results of a more general notion can be found e.g. in \cite{MMN}.
\begin{defin}
A CR--manifold $(\rm M,\mathcal{D}, {\rm J})$  of type $(n,k)$ is \emph{Levi degenerate} in $x\in \rm M$ if $\dim\emph{Null}_x(\mathcal{L}^\R)\neq 0$.  Levi degenerate CR--manifolds with  $\emph{Null}(\mathcal{ L}^\R)=\mathcal{D}$ are called \emph{Levi flat}.
\end{defin}

\par
 On any Levi flat CR manifold $\rm M$  there is a unique foliation $\mathcal{F}$  by complex manifolds such that $T\mathcal{F}$ is the Levi distribution 
$\emph{Null}_x(\mathcal{L}^\R)$ of $\rm M$,  see \cite{DragCRFoliation}. 
This can be extended to the case when $({\rm M},{\rm J},\mathcal D)$ is Levi degenerate with $\dim_\R\emph{Null}_x(\mathcal{L}^\R)=2n'=
\emph{cost}$ ($2n'\leq 2n$). Since $\emph{Null}
_x(\mathcal{L}^\R)$ is involutive and ${\rm J}$ invariant, one can applies both Frobenius  
and Newlander-Niremberg to obtain a foliation $\mathcal{F}$ of $\rm M$ by complex $n'-$dimensional 
manifolds,  i.e.  $\rm M$ can be realized as a disjoint union $\rm{M}=\bigcup_\alpha {\rm F}_{\alpha}$ by complex $n'-$dimensional 
manifolds ${\rm F}_\alpha$ with $T{\rm F}_\alpha\simeq\emph{Null}_x(\mathcal{L}^\R)$. This is essentially the following:

\begin{theorem}[M. Freeman \cite{Freeman74}]\label{freeman}
Let $\mathcal D$ and $\emph{Null}(\mathcal{L}^\R)$ have constant dimension on $\rm M$. Then for each $x\in {\rm M}$ there exist a neighborhood $U$ of $x$ and a unique smooth foliation of ${\rm M}\cap U$ by complex submanifolds, such that the tangent space to the leaf through $x$ is $\emph{Null}_x(\mathcal{L}^\R)$.
\end{theorem}
Such foliation is called the \emph{Levi foliation} of $\rm M$. The following example shows a Levi flat CR--manifold admitting a Levi foliation.

\begin{example}[E. Cartan,  \cite{Cartan1933}]\label{cartan}
Let  $\rm M$  be  a Levi-flat  real-analytic smooth hypersurface in $\C^m$.  It was shown by E. Cartan that near each point $x \in \rm M$, there exist local holomorphic coordinates $( z, w) \in \C^{m-1}\times\C$ vanishing at $x$, such that $\rm M$ near $x$ is described by ${\rm Im}\,w = 0$.  The distribution $\mathcal{D}$ can be given as the tangent space to the leaves of the Levi foliation $\{( z, w) : w = t\,|\, t\in\R\}$. 
\end{example}

We are interested in CR--structures induced by a Riemannian immersion into a K\"ahler manifold. Let  $\tilde {\rm M}$ be a  $m$-dimensional K\"ahler manifold with complex structure $\rm J$ and K\"ahler metric $g$. Following \cite{chenI}, an isometric Riemannian submanifold ${\rm M}$ of ${\rm \tilde M}$ is a {\em CR--submanifold} if there exists a differentiable distribution $\mathcal D\!:{\rm M}\rightarrow T{\rm M}$ of constant dimension such that:

\begin{enumerate}
\item $\mathcal D$ is holomorphic, i.e. ${\rm J}\mathcal D_x=\mathcal D_x$ for each $x\in \rm M$;
\item
 the complementary orthogonal distribution $\mathcal D^\perp: x\mapsto \mathcal 
 D_x^\perp\subset T_x{\rm M}$ is totally real, i.e. ${\rm J} \mathcal D_x^\perp\subset T_x^
 \perp {\rm M}$, where $T_x^\perp \rm M$ is the normal space of $\rm M$ in $\tilde {\rm 
 M}$ at $x$.
 \end{enumerate}
 Observe that the triple $({\rm M},\mathcal D, {\rm J})$ is a CR--manifold in the sense above.

Any isometric Riemannian submanifold of a K\"ahler manifold carries a natural structure of {CR--submanifold} as follows. Define  $\mathcal{D}={\rm J} T{\rm M}\cap T{\rm M} $ and let $\mathcal{D}^\perp$ be the 
orthogonal complement of $\mathcal{D}$ in $T{\rm M}$ with respect to $g$,  i.e. $T_x
{\rm M}=\mathcal{D}_x\oplus \mathcal{D}_x^\perp$. Then, $\mathcal D$ is holomorphic by definition, and ${\rm J} \mathcal D_x^\perp\subset T_x^ \perp {\rm M}$ for $g\left({\rm J}\mathcal{D}^\perp_x,T_x{\rm M}\right)= g\left({\rm J}\left(({\rm J}T_x {\rm M})^\perp\cap T_x {\rm M}\right),T_x {\rm M}\right)=0$. 

Finally, we recall the definition of CR--product (see \cite{chenI,chenII} for details and further results).
\begin{defin}\label{crproducts}
A CR--submanifold of a K\"ahler manifold is called a {\em CR--product} if it is locally a Riemann product of a holomorphic submanifold and a totally real one.
\end{defin}
\begin{remark}Observe that our choice of $\mathcal D=\langle\frac{\partial}{\partial z_1},\dots,\frac{\partial}{\partial z_{m-1}}\rangle_\C$ in Example \ref{cartan} gives $\rm M$ a structure of CR--product with the complex structure and the flat metric induced by the canonical ones in $\mathds C^m$, where $\mathcal{D}^\perp=\langle\frac{\partial}{\partial t}\rangle_\R$.
\end{remark}

To conclude this section we present an example of two complex manifolds which are CR--relatives but not relatives.
\begin{example}[The $3$-sphere in $\mathds C^2$ and in $ \mathds C{\rm P}^2$]\label{3sphere}\rm
We describe here how the natural inclusions of the $3$-sphere $\mathds S^3=\{(z_1,z_2)\in \mathds C^2|\ |z_1|^2+|z_2|^2=1\}$ in  $\mathds C^2$ and $\mathds C{\rm P}^2$ induce the same structure of  CR--submanifold.  Let us start with $\mathds C^2$. Consider the usual identification $\mathds C^2\simeq \mathds R^4$, $(z_1,z_2)\mapsto (x_1,y_1,x_2,y_2)$, where $z_1=x_1+iy_1$, $z_2=x_2+iy_2$. Let $p=(z_1,z_2)=(x_1,y_1,x_2,y_2)\in \mathds S^3$. 
Denote by $X^\perp$ the normal vector, i.e.:
\begin{equation}\label{xperp}
X^\perp:=x_1\frac{\partial}{\partial x_1}+y_1\frac{\partial}{\partial y_1}+x_2\frac{\partial}{\partial x_2}+y_2\frac{\partial}{\partial y_2}.
\end{equation}
Then a basis for $T_p\mathds S^3$ is given by $\{V_1,V_2,V_3\}$ with:
$$
V_1=y_1\frac{\partial}{\partial x_1}-x_1\frac{\partial}{\partial y_1}+y_2\frac{\partial}{\partial x_2}-x_2\frac{\partial}{\partial y_2}, \quad\quad
V_2=x_2\frac{\partial}{\partial x_1}-y_2\frac{\partial}{\partial y_1}-x_1\frac{\partial}{\partial x_2}+y_1\frac{\partial}{\partial y_2},
$$
$$
V_3=-y_2\frac{\partial}{\partial x_1}-x_2\frac{\partial}{\partial y_1}+y_1\frac{\partial}{\partial x_2}+x_1\frac{\partial}{\partial y_2},
$$

The complex structure ${\rm J}$ acts on the $V_j$'s by ${\rm J}V_1=X^\perp$, ${\rm J}V_2= -V_3$, ${\rm J}V_3=V_2$,  thus $\mathcal D:={ T}\mathds S^3\cap {\rm J}({ T}\mathds S^3)$ is spanned by $V_2$ and $V_3$. Then $(\mathds S^3, \mathcal{D}, \rm J)$ has a natural structure of CR-- submanifold of $\C^2$.

 Consider now homogeneous coordinates $[Z_0:Z_1:Z_2]$ on $\mathds C{\rm P}^2$ and define affine coordinates $z_1=\frac{Z_1}{Z_0}$, $z_2=\frac{Z_2}{Z_0}$ on the open set $U_0=\{Z_0\neq 0\}$. The $3$-sphere $\mathds S^3\subset U_0$ is realized as CR--submanifold of $\mathds C{\rm P}^2$ in the following way. The Fubini--Study metric $g_{FS}$ is described on $U_0$ by the K\"ahler potential $\Phi(z)=\log\left(1+|z_1|^2+|z_2|^2\right)$ (see \eqref{phib} above). Let $X^\perp$ be as in \eqref{xperp}, then it is not hard to see that $g_{FS}(X^\perp,\cdot)=g_0(X^\perp, \cdot)$, where $g_0$ is the euclidean metric on $U_0$. Thus, 
a basis for $T_p\mathds S^3$ is given again by $\{V_1,V_2,V_3\}$ as in the $\mathds C^2$ case, and  $\mathcal D$ is spanned by $V_2$ and $V_3$.
We conclude  that the induced structure of CR--submanifold is the same on $\mathds S^3$ and it can be seen that it is Levi non degenerate. 
\end{example}

\section{Proof of Theorem \ref{main} and corollaries}

\begin{proof}[Proof of Theorem \ref{main}]\label{proofmain}
Let $(\tilde {\rm M}_1,{\rm J}_1,g_1)$, $(\tilde{ \rm M}_2,{\rm J}_2,g_2)$ be two K\"ahler manifolds which are not relatives and let $\rm M$ be a common CR--submanifold. 
Denote by $f_1\!:{\rm M}\rightarrow  {\rm\tilde{\rm  M}}_1$, $f_2\!:{\rm M}\rightarrow 
\tilde{\rm  M}_2$  the isometric immersions and by $(\mathcal D,{\rm J},g)$ the 
holomorphic distribution, the almost complex structure and the Riemannian metric on 
${\rm M}$  induced by both $\tilde{\rm  M}_1$ and $\tilde{\rm  M}_2$, i.e. ${\rm J}
_j:={f_j}_*{\rm J}$ and $f_j^*g_j=g$,   for $j=1,2$.  Observe that $\dim \mathcal D_x^
\perp\neq 0$, otherwise $\rm M$ would be a common K\"ahler submanifold of $\tilde 
{\rm M}_1$ and $\tilde{\rm M}_2$. By Th. \ref{freeman} if ${\rm M}$ is degenerate with constant Levi kernel, then it carries a holomorphic foliation $\mathcal F$ with leaves ${\rm M}_c$ such that $({\rm M}_c, {\rm J}|_{{\rm M_c}})$ is a complex submanifold of ${\rm M}$.
 It remains to shows that the metric induced by the inclusions of 
${\rm M}_c\subset {\rm M}$ in $\tilde{ \rm M}_1$ and $\tilde{\rm M}_2$ induce the same K\"ahler structure 
on ${\rm M_c}$. The 
map $i_j=f_j\circ i_c\!: {\rm M}_c\rightarrow \tilde{ \rm M}_j$, for $j=1,2$, where $i_c$ is the 
inclusion $i_c\!: {\rm M}_c\rightarrow {\rm M}$,  is holomorphic, for 
\begin{equation}\label{igei}
{i_j}_* {\rm J}|_{{\rm M}_c}(X)={f_j}_* {i_c}_* {\rm J}|_{{\rm M}_c}(X)={f_j}_*{\rm J}(X)={\rm J}_j(X)={\rm J}_j({i_j}_*X)\nonumber
\end{equation}
 for any $X\in {\mathcal Null}_x(\mathcal L^{\mathds R})$.  Since $i_1^*g_1=i_c^*g=i_2^*g_2$ on ${\rm M}_c$, we need only to show that $({\rm M}_c {\rm J}|_{{\rm M}_c},i^*_cg)$ is K\"ahler. Let us first show that it is hermitian. Since $i_c^*g({\rm J}|_{{\rm M}_c}X,{\rm J}|_{{\rm M}_c}Y)=g({i_c}_*{\rm J}|_{{\rm M}_c}X,{i_c}_*{\rm J}|_{{\rm M}_c}Y)=g({\rm J}X,{\rm J}Y)=g(X,Y)$ for any $X$, $Y\in {\mathcal Null}_x(\mathcal L^{\mathds R})$, where we used that that ${\mathcal Null}_x(\mathcal L^{\mathds R})$ is ${\rm J}|_{{\rm M}_c}$ invariant. Conclusion follows by finally observing that (1,1)-form $\omega_c(X,Y):=i^*_jg({\rm J}|_{{\rm M_c}}X,Y)$ is the pull-back through the holomorphic map $i_j$ of the K\"ahler form $\omega_j$ on $\tilde{\rm M}_j$ and thus it is closed.
\end{proof}

\begin{proof}[Proof of Corollary \ref{cor2}]
A CR--product is Levi flat, thus the integrable manifold of its distribution $\mathcal D$ is a holomorphic submanifold. Further, $\mathcal D={\mathcal Null}(\mathcal L^{\mathds R})$ has constant dimension. Conclusion follows by Theorem \ref{main}.
\end{proof}

\begin{proof}[Proof of Corollary \ref{cor}]
Let $({\rm M},g)$ be a flag manifold and let $\omega$ be its K\"ahler form. Since a flag manifold is a homogeneous variety, $\omega$ is integral up to rescaling. If ${\rm M}$ is not simply connected, Theorem \ref{loimossa} applies locally, proving that there exists $\lambda>0$ such that $({\rm M},\lambda g)$ admits a local K\"ahler immersion in $\mathds C{\rm P}^m$. We claim that $\frac1{\sqrt{\lambda}}f$ is a local K\"ahler immersion of $({\rm M}, g)$ in $\mathds C{\rm P}^m_{\lambda}$. 
If the claim holds true, since by Theorem \ref{umehara} combined with Calabi Rigidity Theorem \ref{calabirigidity} $\mathds C^{m'}$ and $\mathds C{\rm P}_{\lambda}^m$ are not relatives for any $m$, $m'$, $\lambda>0$, conclusion follows by Theorem \ref{main}.  In order to prove the claim observe that locally if we denote by $\varphi$ a K\"ahler potential for $g$, since $f^*\omega_{FS}=\omega$, by \eqref{phib} we have:
$$
f^*\Phi_1=\log\left(1+ \sum_{j=1}^m|f_j|^2\right)=\lambda( \varphi+h+\bar h),
$$
for some holomorphic function $h$ on $U$, and conclusion follows by:
$$
\left(\frac1{\sqrt{\lambda}}f\right)^*\Phi_{\lambda}=\frac1{\lambda} \log\left(1+ \sum_{j=1}^m|f_j|^2\right)= \varphi+h+\bar h.
$$
\end{proof}

\end{document}